\documentclass[a4paper]{scrartcl}
\usepackage[utf8]{inputenc}
\usepackage{lmodern}

\usepackage{amsthm,amsmath,amssymb,amstext,amsfonts,amsxtra}

\newcommand{\field}[1]{\mathbb{#1}}

\newcommand{\N}{\field{N}}

\newcommand{\C}{\field{C}}

\numberwithin{equation}{section}
\newtheorem{theorem}{Theorem}[section]
\newtheorem{lemma}[theorem]{Lemma}

\theoremstyle{remark}

\renewenvironment{proof}[1][Proof]{\begin{trivlist}
\item[\hskip \labelsep {\bfseries #1:}]}{\qed\end{trivlist}}

\title{A combinatorial proof and refinement of a partition identity of Siladi\'c}
\author{Jehanne Dousse}
\date{\today}

\begin{document}

\maketitle

%
%

\begin{abstract}
In this paper we give a combinatorial proof and refinement of a Rogers-Ramanujan type partition identity of Siladi\'c~\cite{Siladic} arising from the study of Lie algebras. Our proof uses generating functions and $q$-difference equations.
\end{abstract}

\section{Introduction}
A partition of $n$ is a non-increasing sequence of natural numbers whose sum is $n$.
For example, there are $5$ partitions of $4$: $4$, $3+1$, $2+2$, $2+1+1$ and $1+1+1+1$.
The Rogers-Ramanujan identities~\cite{RogersRamanujan}, first discovered by Rogers in 1894 and rediscovered by Ramanujan in 1917 are the following $q$-series identities:
\begin{theorem}
\label{rr}
Let $a=0$ or $1$. Then
\begin{equation*}
\sum_{k=0}^n \frac{q^{n(n+a)}}{(1-q)(1-q^2)...(1-q^n)} = \prod_{k=0}^{\infty} \frac{1}{(1-q^{5k+a+1})(1-q^{5k+4-a})}.
\end{equation*}
\end{theorem}
These analytic identities can be interpreted in terms of partitions in the following way:
\begin{theorem}
Let $a=0$ or $1$. Then for every natural number $n$, the number of partitions of $n$ such that the difference between two consecutive parts is at least $2$ and the part $1$ appears at most $1-a$ times is equal to the number of partitions of $n$ into parts congruent to $\pm (1+a) \mod 5.$
\end{theorem}
Rogers-Ramanujan type partition identities establish equalities between certain types of partitions with difference conditions and partitions whose generating functions is an infinite product.

Since the 1980's, many connections between representations of Lie algebras and Rogers-Ramanujan type partition identities have emerged. Lepowsky and Wilson~\cite{Lepowsky} were the first to establish this link by giving an interpretation of Theorem~\ref{rr} in terms of representations of the affine Lie algebra $\mathfrak{sl}(2,\C)\sptilde.$ Similar methods were subsequently applied to other representations of affine Lie algebras, yielding new partition identities of the Rogers-Ramanujan type, as those discovered by Capparelli~\cite{Capparelli}, Primc~\cite{Primc} and Meurman-Primc~\cite{Meurman}. Capparelli's conjecture was proved combinatorially by Andrews in~\cite{Alladi} and~\cite{Andrews1} just before Capparelli finished proving them with Lie-algebraic techniques.
However, most of the Rogers-Ramanujan type partition identities arising from the study of Lie algebras have yet to be understood combinatorially.

In \cite{Siladic}, Siladi\'c proved the following theorem by studying representations of the twisted affine Lie algebra $A^{(2)}_2.$

\begin{theorem}
\label{siladic}
The number of partitions $\lambda_1 + ...+ \lambda_s$ of an integer $n$ into parts different from $2$ such that difference between two consecutive parts is at least $5$ (ie. $\lambda_i - \lambda_{i+1} \geq 5$) and
\begin{equation*}
\begin{aligned}
&\lambda_i - \lambda_{i+1} = 5 \Rightarrow \lambda_i + \lambda_{i+1} \not \equiv \pm 1, \pm 5, \pm 7 \mod 16,
\\&\lambda_i - \lambda_{i+1} = 6 \Rightarrow \lambda_i + \lambda_{i+1} \not \equiv \pm 2, \pm 6 \mod 16,
\\&\lambda_i - \lambda_{i+1} = 7 \Rightarrow \lambda_i + \lambda_{i+1} \not \equiv \pm 3 \mod 16,
\\&\lambda_i - \lambda_{i+1} = 8 \Rightarrow \lambda_i + \lambda_{i+1} \not \equiv \pm 4 \mod 16,
\end{aligned}
\end{equation*}
is equal to the number of partitions of $n$ into distinct odd parts.
\end{theorem}

This paper is devoted to proving combinatorially and refining Theorem~\ref{siladic}. In Section~\ref{reformulation} we give an equivalent formulation of Theorem~\ref{siladic} which is easier to manipulate in terms of partitions. In Section~\ref{qdiff} we establish $q$-difference equations satisfied by the generating functions of partitions considered in Theorem~\ref{siladic}. Finally, we use those $q$-difference equations to prove Theorem~\ref{siladic} by induction.

Our refinement of Theorem~\ref{siladic} is the following:

\begin{theorem}
\label{refinement}
For $n \in \N$ and $k \in \N^*$, let $A(k,n)$ denote the number of partitions $\lambda_1 + ...+ \lambda_s$ of $n$ such that $k$ equals the number of odd part plus twice the number of even parts, satisfying the following conditions:
\begin{enumerate}
  \item $\forall i \geq 1, \lambda_i \neq 2$,
  \item $\forall i \geq 1, \lambda_i - \lambda_{i+1} \geq 5$,
  \item $\forall i \geq 1$,
  \begin{equation*}
  \begin{aligned}
&\lambda_i - \lambda_{i+1} = 5 \Rightarrow \lambda_i \equiv 1, 4 \mod 8,
\\&\lambda_i - \lambda_{i+1} = 6 \Rightarrow \lambda_i \equiv 1, 3, 5, 7 \mod 8,
\\&\lambda_i - \lambda_{i+1} = 7 \Rightarrow \lambda_i \equiv 0, 1, 3, 4, 6, 7 \mod 8,
\\&\lambda_i - \lambda_{i+1} = 8 \Rightarrow \lambda_i \equiv 0, 1, 3, 4, 5, 7 \mod 8.
	\end{aligned}
	\end{equation*}
\end{enumerate}

For $n \in \N$ and $k \in \N^*$, let $B(k,n)$ denote the number of partitions of $n$ into $k$ distinct odd parts.
Then for all $n \in \N$ and $k \in \N^*$, $A(k,n)=B(k,n)$.
\end{theorem}

\section{Reformulating the problem}
\label{reformulation}
Our idea is to find $q$-difference equations and use them to prove Theorem~\ref{siladic}, but its original formulation is not very convenient to manipulate combinatorially because it gives conditions on the sum of two consecutive parts of the partition. Therefore we will transform those conditions into conditions that only involve one part at a time.

\begin{lemma}
\label{equiv}
Conditions 
\begin{equation}
\label{cond5i}
\lambda_i - \lambda_{i+1} = 5 \Rightarrow \lambda_i + \lambda_{i+1} \not \equiv \pm 1, \pm 5, \pm 7 \mod 16,
\end{equation}
\begin{equation}
\label{cond6i}
\lambda_i - \lambda_{i+1} = 6 \Rightarrow \lambda_i + \lambda_{i+1} \not \equiv \pm 2, \pm 6 \mod 16,
\end{equation}
\begin{equation}
\label{cond7i}
\lambda_i - \lambda_{i+1} = 7 \Rightarrow \lambda_i + \lambda_{i+1} \not \equiv \pm 3 \mod 16,
\end{equation}
\begin{equation}
\label{cond8i}
\lambda_i - \lambda_{i+1} = 8 \Rightarrow \lambda_i + \lambda_{i+1} \not \equiv \pm 4 \mod 16,
\end{equation}
are respectively equivalent to conditions
\begin{equation}
\label{cond5}
\lambda_i - \lambda_{i+1} = 5 \Rightarrow \lambda_i \equiv 1, 4 \mod 8,
\end{equation}
\begin{equation}
\label{cond6}
\lambda_i - \lambda_{i+1} = 6 \Rightarrow \lambda_i \equiv 1, 3, 5, 7 \mod 8,
\end{equation}
\begin{equation}
\label{cond7}
\lambda_i - \lambda_{i+1} = 7 \Rightarrow \lambda_i \equiv 0, 1, 3, 4, 6, 7 \mod 8,
\end{equation}
\begin{equation}
\label{cond8}
\lambda_i - \lambda_{i+1} = 8 \Rightarrow \lambda_i \equiv 0, 1, 3, 4, 5, 7 \mod 8.
\end{equation}
\end{lemma}

\begin{proof}
Let us prove the first equivalence. The others are proved in exactly the same way.
We have
\begin{equation*}
\begin{aligned}
&\lambda_i - \lambda_{i+1} = 5 \Rightarrow \lambda_i + \lambda_{i+1} \not \equiv \pm 1, \pm 5, \pm 7 \mod 16
\\ \Leftrightarrow~~&\lambda_i - \lambda_{i+1} = 5 \Rightarrow \lambda_i + \lambda_{i+1} \not \equiv 1, 15, 5, 11, 7, 9 \mod 16
\\ \Leftrightarrow~~&\lambda_i - \lambda_{i+1} = 5 \Rightarrow 2 \lambda_i = \lambda_i + \lambda_{i+1} + \lambda_i - \lambda_{i+1} \not \equiv 6, 4, 10, 0, 12, 14 \mod 16
\\ \Leftrightarrow~~&\lambda_i - \lambda_{i+1} = 5 \Rightarrow \lambda_i \not \equiv 3, 2, 5, 0, 6, 7 \mod 8
\\ \Leftrightarrow~~&\lambda_i - \lambda_{i+1} = 5 \Rightarrow \lambda_i \equiv 1, 4 \mod 8.
\end{aligned}
\end{equation*}
Therefore condition~\eqref{cond5i} is equivalent to condition~\eqref{cond5}.
\end{proof}

By Lemma~\ref{equiv}, Theorem~\ref{siladic} is equivalent to the following theorem.

\begin{theorem}
The number of partitions $\lambda_1 + ...+ \lambda_s$ of an integer $n$ into parts different from $2$ such that difference between two consecutive parts is at least $5$ (i.e.. $\lambda_i - \lambda_{i+1} \geq 5$) and
\begin{equation*}
\begin{aligned}
&\lambda_i - \lambda_{i+1} = 5 \Rightarrow \lambda_i \equiv 1, 4 \mod 8,
\\&\lambda_i - \lambda_{i+1} = 6 \Rightarrow \lambda_i \equiv 1, 3, 5, 7 \mod 8,
\\&\lambda_i - \lambda_{i+1} = 7 \Rightarrow \lambda_i \equiv 0, 1, 3, 4, 6, 7 \mod 8,
\\&\lambda_i - \lambda_{i+1} = 8 \Rightarrow \lambda_i \equiv 0, 1, 3, 4, 5, 7 \mod 8,
\end{aligned}
\end{equation*}
is equal to the number of partitions of $n$ into distinct odd parts.
\end{theorem}

Moreover for every $n$, the sets of partitions are exactly the same as those in Theorem~\ref{siladic}, so this is just a reformulation of the same theorem.

\section{Obtaining $q$-difference equations}
\label{qdiff}
Now that we have stated Theorem~\ref{siladic} in a  more convenient manner, we can establish our $q$-difference equations and prove Theorem~\ref{refinement}.

For $n \in N$, $k \in \N^*$, let $a_N(k,n)$ denote the number of partitions $\lambda_1+...+\lambda_s$ counted by $A(k,n)$ such that the largest part $\lambda_1$ is at most $N$.
Let also $e_N(k,n)$ denote the number of partitions $\lambda_1+...+\lambda_s$ counted by $A(k,n)$ such that the largest part $\lambda_1$ is equal to $N$.
We define, for $|t|<1$, $|q|<1$, $N \in \N^*$, 
\begin{equation*}
G_N (t,q) = 1+ \sum_{k=1}^{\infty} \sum_{n=1}^{\infty} a_N(k,n) t^k q^n.
\end{equation*}
Thus $G_{\infty}(t,q)= \lim_{N \rightarrow \infty} G_N (t,q)$ is the generating function for the partitions counted by $A(k,n)$.

Our goal is to show that $$\forall N \in \N^*,  G_{2N}(t,q)=(1+tq)G_{2N-3}(tq^2,q).$$ Indeed we can then let $N$ go to infinity and deduce $$G_{\infty}(t,q)=(1+tq)G_{\infty}(tq^2,q)=(1+tq)(1+tq^3)G_{\infty}(tq^4,q)=...,$$ which means that $$G_{\infty}(t,q)= \prod_{k=0}^{\infty} \left(1+ tq^{2k+1}\right),$$ which is the generating function for partitions counted by $B(k,n).$

Let us now state some $q$-difference equations that we will use throughout our proof in Section~\ref{induction}. We have the following identities:

\begin{lemma}
\label{eqd}
For all $k,n,N \in \N^*,$
\begin{equation}
\label{eqd1}
a_{8N}(k,n)=a_{8N-1}(k,n)+a_{8N-7}(k-2,n-8N),
\end{equation}
\begin{equation}
\label{eqd2}
a_{8N+1}(k,n)=a_{8N}(t,q)+a_{8N-4}(k-1,n-(8N+1)),
\end{equation}
\begin{equation}
\label{eqd3}
a_{8N+2}(k,n)=a_{8N+1}(k,n)+a_{8N-7}(k-2,n-(8N+2)),
\end{equation}
\begin{equation}
\label{eqd4}
a_{8N+3}(k,n)=a_{8N+2}(k,n)+a_{8N-3}(k-1,n-(8N+3)),
\end{equation}
\begin{equation}
\label{eqd5}
a_{8N+4}(k,n)=a_{8N+3}(k,n)+a_{8N-3}(k-2,n-(8N+4))+a_{8N-7}(k-3,n-(16N+3)),
\end{equation}
\begin{equation}
\label{eqd6}
a_{8N+5}(k,n)=a_{8N+4}(k,n)+a_{8N-3}(k-1,n-(8N+5))+a_{8N-7}(k-2,n-(16N+4)),
\end{equation}
\begin{equation}
\label{eqd7}
a_{8N+6}(k,n)=a_{8N+5}(k,n)+a_{8N-3}(k-2,n-(8N+6))+a_{8N-7}(k-3,n-(16N+5)),
\end{equation}
\begin{equation}
\label{eqd8}
a_{8N+7}(k,n)=a_{8N+6}(k,n)+a_{8N+1}(k-1,n-(8N+7)).
\end{equation}
\end{lemma}

\begin{proof}
We prove equations~\eqref{eqd1} and~\eqref{eqd5}. Equations~\eqref{eqd2},~\eqref{eqd3},~\eqref{eqd4} and~\eqref{eqd8} are proved in the same way as equation~\eqref{eqd1}, and equations~\eqref{eqd6} and~\eqref{eqd7} in the same way as equation~\eqref{eqd5}.

Let us prove \eqref{eqd1}.We divide the set of partitions enumerated by $a_{8N}(k,n)$ into two sets, those with largest part less than $8N$ and those with largest part equal to $8N$. Thus
\begin{equation*}
a_{8N}(k,n)=a_{8N-1}(k,n)+e_{8N}(k,n).
\end{equation*}
Let us now consider a partition $\lambda_1+ \lambda_2 + ... + \lambda_s$ counted by $e_{8N}(k,n).$ By Conditions~\eqref{cond5}-\eqref{cond8}, $\lambda_1 - \lambda_2 \geq 7$, therefore $\lambda_2 \leq 8N-7.$ Let us remove the largest part $\lambda_1=8N$. The largest part is now $\lambda_2 \leq 8N-7$, the number partitioned is $n-8N$, and we removed an even part so $k$ becomes $k-2$. We obtain a partition counted by $a_{8N-7}(k-2,n-8N).$ This process is reversible, because we can add a part equal to $8N$ to any partition counted by $a_{8N-7}(k-2,n-8N)$ and obtain a partition counted by $e_{8N}(k,n)$ so we have a bijection between partitions counted by $e_{8N}(k,n)$ and those counted by $a_{8N-7}(k-2,n-8N).$
Therefore $$e_{8N}(k,n)=a_{8N-7}(k-2,n-8N)$$ for all $k,n,N \in \N^*$ and~\eqref{eqd1} is proved.

Let us now prove~\eqref{eqd5}.
Again let us divide the set of partitions enumerated by $a_{8N+4}(k,n)$ into two sets, those with largest part less than $8N+4$ and those with largest part equal to $8N+4$. Thus
\begin{equation*}
a_{8N+4}(k,n)=a_{8N+3}(k,n)+e_{8N+4}(k,n).
\end{equation*}
Let us now consider a partition $\lambda_1+ \lambda_2 + ... + \lambda_s$ counted by $e_{8N+4}(k,n).$ By Conditions~\eqref{cond5}-\eqref{cond8}, $\lambda_1 - \lambda_2 = 5$ or $\lambda_1 - \lambda_2 \geq 7$, therefore $\lambda_2=8N-1$ or $\lambda_2 \leq 8N-3$. Let us remove the largest part $\lambda_1=8N+4$. If $\lambda_2=8N-1$, we obtain a partition counted by $e_{8N-1}(k-1,n-(8N+5))$. If $\lambda_2 \leq 8N-3$, we obtain a partition counted by $a_{8N-3}(k-1,n-(8N+4)).$ This process is also reversible and the following holds:
\begin{equation*}
e_{8N+4}(k,n)= e_{8N-1}(k-1,n-(8N+4)) + a_{8N-3}(k-1,n-(8N+4))
\end{equation*}
Moreover, again by removing the largest part, we can prove that $$e_{8N-1}(k-1,n-(8N+4))=a_{8N-7}(k-2,n-(16N+3)).$$
This concludes the proof of~\eqref{eqd5}.
\end{proof}

The equations of Lemma~\ref{eqd} lead to the following $q$-difference equations:

\begin{lemma}
For all $N \in \N^*,$
\label{equations}
\begin{equation}
\label{eq1}
G_{8N}(t,q)=G_{8N-1}(t,q)+t^{2}q^{8N} G_{8N-7}(t,q),
\end{equation}
\begin{equation}
\label{eq2}
G_{8N+1}(t,q)=G_{8N}(t,q)+tq^{8N+1} G_{8N-4}(t,q),
\end{equation}
\begin{equation}
\label{eq3}
G_{8N+2}(t,q)=G_{8N+1}(t,q)+t^{2}q^{8N+2} G_{8N-7}(t,q),
\end{equation}
\begin{equation}
\label{eq4}
G_{8N+3}(t,q)=G_{8N+2}(t,q)+tq^{8N+3} G_{8N-3}(t,q),
\end{equation}
\begin{equation}
\label{eq5}
G_{8N+4}(t,q)=G_{8N+3}(t,q)+t^2q^{8N+4} G_{8N-3}(t,q)+t^3q^{16N+3}G_{8N-7}(t,q),
\end{equation}
\begin{equation}
\label{eq6}
G_{8N+5}(t,q)=G_{8N+4}(t,q)+tq^{8N+5} G_{8N-3}(t,q)+t^2q^{16N+4}G_{8N-7}(t,q),
\end{equation}
\begin{equation}
\label{eq7}
G_{8N+6}(t,q)=G_{8N+5}(t,q)+t^2q^{8N+6} G_{8N-3}(t,q)+t^3q^{16N+5}G_{8N-7}(t,q),
\end{equation}
\begin{equation}
\label{eq8}
G_{8N+7}(t,q)=G_{8N+6}(t,q)+tq^{8N+7} G_{8N+1}(t,q).
\end{equation}
\end{lemma}

Some more $q$-difference equations will be stated in the proof of Section~\ref{induction} as their interest arises from the proof itself.

Even if we use the idea of counting certain parts twice as in Andrews' proof of Schur's theorem~\cite{Andrews3} and the author's proof of Schur's theorem for overpartitions~\cite{Dousse}, the consequent number of equations (we have $8$ equations here while there were only $3$ equations in the proofs above mentioned) make it difficult to find directly a recurrence equation satisfied by $G_{8N}(t,q)$ and use the same method. Therefore we proceed differently as shown in next section.

\section{Proof of Theorem~\ref{refinement}}
\label{induction}
In this section we prove the following theorem by induction:

\begin{theorem}
\label{main}
For all $m \in \N^*,$
\begin{equation}
\label{2N}
G_{2m}(t,q)=(1+tq)G_{2m-3}(tq^2,q).
\end{equation}
\end{theorem}

\subsection{Initialisation}
First we need to check some initial cases.

With the initial conditions
\begin{equation*}
G_0(t,q)=1,
\end{equation*}
\begin{equation*}
G_1(t,q)=1+tq,
\end{equation*}
\begin{equation*}
G_2(t,q)=1+tq,
\end{equation*}
\begin{equation*}
G_3(t,q)=G_2(t,q)+tq^3,
\end{equation*}
\begin{equation*}
G_4(t,q)=G_3(t,q)+t^2q^4,
\end{equation*}
\begin{equation*}
G_5(t,q)=G_4(t,q)+tq^5,
\end{equation*}
\begin{equation*}
G_6(t,q)=G_5(t,q)+t^2q^6,
\end{equation*}
\begin{equation*}
G_7(t,q)=G_6(t,q)+tq^7+t^2q^8,
\end{equation*}
and equations~\eqref{eq1}-\eqref{eq8}, we use MAPLE to check that Theorem~\ref{main} is verified for $m=1,...,8.$

Let us now assume that Theorem~\ref{main} is true for all $k \leq m-1$ and show that equation~\eqref{2N} is also satisfied for $m$. To do so, we will consider $4$ different cases: $m \equiv 0 \mod 4,$ $m \equiv 1 \mod 4,$ $m \equiv 2 \mod 4$ and $m \equiv 3 \mod 4.$

\subsection{First case: $m \equiv 0 \mod 4$}
We start by studying the case where $m=4N$ with $N \geq 2.$ We want to prove that $$G_{8N}(t,q)=(1+tq)G_{8N-3}(tq^2,q).$$

Replacing $N$ by $N-1$ in~\eqref{eq8} and substituting into~\eqref{eq1}, we obtain
\begin{equation}
\label{plic}
G_{8N}(t,q)=G_{8N-2}(t,q)+\left(tq^{8N-1}+t^2q^{8N}\right)G_{8N-7}(t,q).
\end{equation}
We now replace $N$ by $N-1$ in~\eqref{eq2} and substitute into~\eqref{plic}. This gives
\begin{equation*}
G_{8N}(t,q)=G_{8N-2}(t,q)+(1+tq)tq^{8N-1}G_{8N-8}(t,q)+(1+tq)t^2q^{16N-8}G_{8N-12}(t,q).
\end{equation*}
Then by the induction hypothesis,
\begin{equation}
\label{ploc}
\begin{aligned}
G_{8N}(t,q)=(1+tq)\big[&G_{8N-5}(tq^2,q)+(1+tq)tq^{8N-1}G_{8N-11}(tq^2,q)
\\&+(1+tq)t^2q^{16N-8}G_{8N-15}(tq^2,q)\big].
\end{aligned}
\end{equation}
Replacing $N$ by $N-1$ and $t$ by $tq^2$ in~\eqref{eq5}, we obtain
\begin{equation}
\label{eq1*}
G_{8N-4}(tq^2,q)=G_{8N-5}(tq^2,q)+t^2q^{8N}G_{8N-11}(tq^2,q)+t^3q^{16N-7}G_{8N-15}(tq^2,q).
\end{equation}
Replacing $N$ by $N-1$ and $t$ by $tq^2$ in~\eqref{eq5} gives
\begin{equation}
\label{eq2*}
G_{8N-3}(tq^2,q)=G_{8N-4}(tq^2,q)+tq^{8N-1}G_{8N-11}(tq^2,q)+t^2q^{16N-8}G_{8N-15}(tq^2,q).
\end{equation}
Adding~\eqref{eq1*} and~\eqref{eq2*}, we get
\begin{equation*}
\begin{aligned}
G_{8N-3}(tq^2,q)&=G_{8N-5}(tq^2,q)+(1+tq)tq^{8N-1}G_{8N-11}(tq^2,q)
\\&+(1+tq)t^2q^{16N-8}G_{8N-15}(tq^2,q).
\end{aligned}
\end{equation*}
Thus by~\eqref{ploc},we deduce that
\begin{equation*}
G_{8N}(t,q)=(1+tq)G_{8N-3}(tq^2,q).
\end{equation*}

It remains now to treat the cases $m \equiv 1,2,3 \mod 4.$

\subsection{Second case: $m \equiv 1 \mod 4$}
We now assume that $m=4N+1$ with $N \geq 2$ and prove that $$G_{8N+2}(t,q)=(1+tq)G_{8N-1}(tq^2,q).$$

Replacing $N$ by $N-1$ in~\eqref{eq8}, we obtain
\begin{equation}
\label{cas2eq1}
G_{8N-1}(t,q)=G_{8N-2}(t,q)+tq^{8N-1} G_{8N-7}(t,q).
\end{equation}
Replacing $N$ by $N-1$ in~\eqref{eq7} and substituting in~\eqref{cas2eq1}, we get
\begin{equation}
\label{cas2eq2}
\begin{aligned}
G_{8N-1}(t,q)&=G_{8N-3}(t,q)+tq^{8N-1} G_{8N-7}(t,q)
\\&+t^2q^{8N-2} G_{8N-11}(t,q) +t^3q^{16N-11} G_{8N-15}(t,q).
\end{aligned}
\end{equation}
Then replacing $t$ by $tq^2$ in~\eqref{cas2eq2}, we obtain the following equation:
\begin{equation}
\label{cas2main}
\begin{aligned}
G_{8N-1}(tq^2,q)&=G_{8N-3}(tq^2,q)+tq^{8N+1} G_{8N-7}(tq^2,q)
\\&+t^2q^{8N+2} G_{8N-11}(tq^2,q) +t^3q^{16N-5} G_{8N-15}(tq^2,q).
\end{aligned}
\end{equation}
Thus we want to prove that
\begin{equation*}
\begin{aligned}
G_{8N+2}(t,q)&=G_{8N}(t,q)+tq^{8N+1} G_{8N-4}(t,q)
\\&+t^2q^{8N+2} G_{8N-8}(t,q) +t^3q^{16N-5} G_{8N-12}(tq^2,q).
\end{aligned}
\end{equation*}
in order to be able to use the induction hypothesis.
We will need a few new equations to do so.

By definition, for all $n,k,N \in \N^*$,
\begin{equation}
\label{triangle2}
a_{8N+2}(k,n)=a_{8N}(k,n)+e_{8N+1}(k,n)+e_{8N+2}(k,n).
\end{equation}
We need formulas for $e_{8N+1}(k,n)$ and $e_{8N+2}(k,n).$

\begin{lemma}
\label{cas2}
For all $n,k,N \in \N^*$,
\begin{equation}
\label{dur1}
e_{8N+1}(k,n)=a_{8N-4}(k-1,n-(8N+1)),
\end{equation}
\begin{equation}
\label{dur2}
e_{8N+2}(k,n)=a_{8N-8}(k-2,n-(8N+2))+a_{8N-12}(k-3,n-(16N-5)),
\end{equation}
\end{lemma}

\begin{proof}
\ 
\begin{itemize}
\item Proof of~\eqref{dur1}:

Let us consider a partition $\lambda_1+ \lambda_2 + ... + \lambda_s$ counted by $e_{8N+1}(k,n).$ By conditions~\eqref{cond5}-\eqref{cond8}, $\lambda_1 - \lambda_2 \geq 5$, therefore $\lambda_2 \leq 8N-4.$ Therefore if we remove the largest part, we obtain a partition counted by $a_{8N-4}(k-1,n-(8N+1)).$

\item Proof of~\eqref{dur2}:

Let us consider a partition $\lambda_1+ \lambda_2 + ... + \lambda_s$ counted by $e_{8N+2}(k,n).$ By conditions~\eqref{cond5}-\eqref{cond8}, $\lambda_1 - \lambda_2 \geq 9$, therefore $\lambda_2 \leq 8N-7.$ Therefore if we remove the largest part, we obtain a partition counted by $a_{8N-7}(k-2,n-(8N+2)).$
So $e_{8N+2}(k,n)= a_{8N-7}(k-2,n-(8N+2)),$ and by definition $$e_{8N+2}(k,n)= a_{8N-8}(k-2,n-(8N+2))+e_{8N-7}(k-2,n-(8N+2)).$$
Let us now consider a partition $\mu_1+ \mu_2 + ... + \mu_r$ counted by $e_{8N-7}(k-2,n-(8N+2)).$ By conditions~\eqref{cond5}-\eqref{cond8}, $\mu_1 - \mu_2 \geq 5$, therefore $\mu_2 \leq 8N-12$. If we remove the largest part $\mu_1=8N-7$, we obtain a partition counted by $a_{8N-12}(k-3,n-(8N+2)-(8N-7)).$
Thus $$e_{8N+2}(k,n)= a_{8N-8}(k-2,n-(8N+2))+a_{8N-12}(k-3,n-(16N-5)).$$
\end{itemize}
\end{proof}

Now by Lemma~\ref{cas2} and~\eqref{triangle2}, for all  $k,n,N \in \N*,$
\begin{equation*}
\begin{aligned}
a_{8N+2}(k,n)&=a_{8N}(k,n)+a_{8N-4}(k-1,n-(8N+1))
\\&+a_{8N-8}(k-2,n-(8N+2))+a_{8N-12}(k-3,n-(16N-5)).
\end{aligned}
\end{equation*}
This leads to the desired $q$-difference equation: 
\begin{equation*}
\begin{aligned}
G_{8N+2}(t,q)&=G_{8N}(t,q)+tq^{8N+1} G_{8N-4}(t,q)
\\&+t^2q^{8N+2} G_{8N-8}(t,q) +t^3q^{16N-5} G_{8N-12}(tq^2,q).
\end{aligned}
\end{equation*}
By the induction hypothesis, the result from the last subsection and~\eqref{cas2main}, we show
$$G_{8N+2}(t,q)=(1+tq)G_{8N-1}(t,q).$$

Let us now turn to the case $m \equiv 2 \mod 4.$

\subsection{Third case: $m \equiv 2 \mod 4$}
We suppose that $m=4N+2$ with $N \geq 2$ and prove that $$G_{8N+4}(t,q)=(1+tq)G_{8N+1}(tq^2,q).$$

Substituting~\eqref{eq1} into~\eqref{eq2}, we have
\begin{equation}
\label{plouf}
G_{8N+1}(t,q)=G_{8N-1}(t,q)+tq^{8N+1} G_{8N-4}(t,q)+t^{2}q^{8N} G_{8N-7}(t,q).
\end{equation}
Replacing $N$ by $N-1$ in~\eqref{eq5} and substituting in~\eqref{plouf}, we have
\begin{equation}
\label{plouf3}
\begin{aligned}
G_{8N+1}(t,q)&=G_{8N-1}(t,q)+tq^{8N+1} G_{8N-5}(t,q)+t^2q^{8N} G_{8N-7}(t,q)
\\&+t^3q^{16N-3} G_{8N-11}(t,q) + t^4q^{24N-12} G_{8N-15}(t,q).
\end{aligned}
\end{equation}
Then replacing $t$ by $tq^2$ in~\eqref{plouf3}, we obtain the following equation:
\begin{equation}
\label{etoile}
\begin{aligned}
G_{8N+1}(tq^2,q)&=G_{8N-1}(tq^2,q)+tq^{8N+3} G_{8N-5}(tq^2,q)+t^2q^{8N+4} G_{8N-7}(tq^2,q)
\\&+t^3q^{16N+3} G_{8N-11}(tq^2,q) + t^4q^{24N-4} G_{8N-15}(tq^2,q).
\end{aligned}
\end{equation}
Thus we want to prove that
\begin{equation*}
\begin{aligned}
G_{8N+4}(t,q)&=G_{8N+2}(t,q)+tq^{8N+3} G_{8N-2}(t,q)+t^2q^{8N+4} G_{8N-4}(t,q)
\\&+t^3q^{16N+3} G_{8N-8}(t,q) + t^4q^{24N-4} G_{8N-12}(t,q).
\end{aligned}
\end{equation*}
Again we need new equations to do so.

By definition, for all $n,k,N \in \N^*$,
\begin{equation}
\label{triangle}
a_{8N+4}(k,n)=a_{8N+2}(k,n)+e_{8N+3}(k,n)+e_{8N+4}(k,n).
\end{equation}
We need formulas for $e_{8N+3}(k,n)$ and $e_{8N+4}(k,n).$

\begin{lemma}
\label{casdifficile}
For all $n,k,N \in \N^*$,
\begin{equation}
\label{dur3}
e_{8N+3}(k,n)=a_{8N-2}(k-1,n-(8N+3))-e_{8N-3}(k-2,n-(8N+4)),
\end{equation}
\begin{equation}
\label{dur4}
\begin{aligned}
e_{8N+4}(k,n)&=a_{8N-4}(k-2,n-(8N+4))+e_{8N-3}(k-2,n-(8N+4))
\\&+a_{8N-8}(k-3,n-(16N+3))+a_{8N-12}(k-4,n-(24N-4)).
\end{aligned}
\end{equation}
\end{lemma}

\begin{proof}
\ 
\begin{itemize}

\item Proof of~\eqref{dur3}:

In the same way as before, by conditions~\eqref{cond5}-\eqref{cond8}, $$e_{N+3}(k,n)=a_{8N-3}(k-1,n-(8N+3)).$$
Thus by definition $$e_{8N+3}(k,n)=a_{8N-2}(k-1,n-(8N+3))-e_{8N-2}(k-1,n-(8N+3)).$$

Now let us consider a partition $\lambda_1+ \lambda_2 + ... + \lambda_s$ counted by $e_{8N-2}(k-1,n-(8N+3))$.
By Conditions~\eqref{cond5}-\eqref{cond8}, $\lambda_1 - \lambda_2 = 7$ or $\lambda_1 - \lambda_2 \geq 9$, therefore $\lambda_2=8N-9$ or $\lambda_2 \leq 8N-11$. Let us remove the largest part $\lambda_1=8N-2$. If $\lambda_2=8N-9$, we obtain a partition counted by $e_{8N-9}(k-3,n-(16N+1))$. If $\lambda_2 \leq 8N-11$, we obtain a partition counted by $a_{8N-11}(k-3,n-(16N+1)).$ Thus the following holds:
\begin{equation*}
\begin{aligned}
e_{8N-2}(k-1,n-(8N+3))&= e_{8N-9}(k-3,n-(16N+1)) 
\\&+ a_{8N-11}(k-3,n-(16N+1)).
\end{aligned}
\end{equation*}

In the exact same way we can show that
\begin{equation*}
\begin{aligned}
e_{8N-3}(k-2,n-(8N+4))&= e_{8N-9}(k-3,n-(16N+1))
\\&+ a_{8N-11}(k-3,n-(16N+1)).
\end{aligned}
\end{equation*}

Therefore $$e_{8N-2}(k-1,n-(8N+3))=e_{8N-3}(k-2,n-(8N+4)),$$ and \eqref{dur3} is proved.

\item Proof of~\eqref{dur4}:

Now let us consider a partition $\lambda_1+ \lambda_2 + ... + \lambda_s$ counted by $e_{8N+4}(k,n)$.
By conditions~\eqref{cond5}-\eqref{cond8}, $\lambda_1 - \lambda_2 = 5$ or $\lambda_1 - \lambda_2 \geq 7$. Therefore by removing the largest part, we obtain
\begin{equation*}
\begin{aligned}
e_{8N+4}(k,n)&= a_{8N-4}(k-2,n-(8N+4))
\\&+e_{8N-3}(k-2,n-(8N+4))+e_{8N-1}(k-2,n-(8N+4)).
\end{aligned}
\end{equation*}

By similar reasoning,
\begin{equation*}
\begin{aligned}
&e_{8N-1}(k-2,n-(8N+4))
\\&= a_{8N-8}(k-3,n-(16+3))+e_{8N-7}(k-3,n-(16N+3))
\\&= a_{8N-8}(k-3,n-(16+3))+a_{8N-12}(k-4,n-(24N-4)).
\end{aligned}
\end{equation*}

Equation~\eqref{dur4} is proved.

\end{itemize}
\end{proof}

Now by Lemma~\ref{casdifficile} and~\eqref{triangle}, for all  $k,n,N \in \N*,$
\begin{equation*}
\begin{aligned}
a_{8N+4}(k,n)&=a_{8N}(k,n)+a_{8N-2}(k-1,n-(8N+3))+a_{8N-4}(k-2,n-(8N+4))
\\&+a_{8N-8}(k-3,n-(16N+3))+a_{8N-12}(k-4,n-(24N-4)).
\end{aligned}
\end{equation*}
This leads to the desired $q$-difference equation: 
\begin{equation*}
\begin{aligned}
G_{8N+4}(t,q)&=G_{8N+2}(t,q)+tq^{8N+3} G_{8N-2}(t,q)+t^2q^{8N+4} G_{8N-4}(t,q)
\\&+t^3q^{16N+3} G_{8N-8}(t,q) + t^4q^{24N-4} G_{8N-12}(t,q).
\end{aligned}
\end{equation*}
By the induction hypothesis and~\eqref{etoile}, we show
$$G_{8N+4}(t,q)=(1+tq)G_{8N+1}(t,q).$$

We can now treat the last case.

\subsection{Fourth case: $m \equiv 3 \mod 4$}

Finally, we suppose that $m=4N+3$ with $N \geq 2$ and prove that $$G_{8N+6}(t,q)=\left(1+tq\right)G_{8N+3}(tq^2,q).$$

Replacing $t$ by $tq^2$ in~\eqref{eq3} and~\eqref{eq4} leads to
\begin{equation}
\label{eq3'}
G_{8N+2}(tq^2,q)=G_{8N+1}(tq^2,q)+t^{2}q^{8N+6} G_{8N-7}(tq^2,q),
\end{equation}
\begin{equation}
\label{eq4'}
G_{8N+3}(tq^2,q)=G_{8N+2}(tq^2,q)+tq^{8N+5} G_{8N-3}(tq^2,q).
\end{equation}
Adding~\eqref{eq3'} and~\eqref{eq4'} we obtain:
\begin{equation}
\label{eq34}
G_{8N+3}(tq^2,q)=G_{8N+1}(tq^2,q)+tq^{8N+5} G_{8N-3}(tq^2,q)+t^{2}q^{8N+6}G_{8N-7}(tq^2,q).
\end{equation}
We now want to show that
\begin{equation*}
G_{8N+6}(t,q)=G_{8N+4}(t,q)+tq^{8N+5} G_{8N}(t,q)+t^{2}q^{8N+6}G_{8N-4}(t,q).
\end{equation*}
By definition we have 
\begin{equation}
\label{pif3}
a_{8N+6}(k,n)=a_{8N+4}(k,n)+e_{8N+5}(k,n)+e_{8N+6}(k,n).
\end{equation}
In a similar manner as above, by conditions~\eqref{cond5}-\eqref{cond8} and removing the largest part, we show that 
\begin{equation}
\label{pif1}
\begin{aligned}
e_{8N+5}(k,n)&=a_{8N}(k-1,n-(8N+5))
\\&-e_{8N}(k-1,n-(8N+5))-e_{8N-2}(k-1,n-(8N+5)),
\end{aligned}
\end{equation}
and 
\begin{equation}
\label{pif2}
\begin{aligned}
e_{8N+6}(k,n)&=a_{8N-4}(k-2,n-(8N+6))
\\&+e_{8N-1}(k-2,n-(8N+6))+e_{8N-3}(k-2,n-(8N+6)).
\end{aligned}
\end{equation}
Yet again by the same method we show that
\begin{equation*}
e_{8N-1}(k-2,n-(8N+6))=a_{8N-7}(k-3,n-(16N+5)),
\end{equation*}
and
\begin{equation*}
e_{8N}(k-1,n-(8N+5))=a_{8N-7}(k-3,n-(16N+5)).
\end{equation*}
Therefore
\begin{equation*}
e_{8N}(k-1,n-(8N+5))=e_{8N-1}(k-2,n-(8N+6)).
\end{equation*}
And in the same way
\begin{equation*}
e_{8N-3}(k-2,n-(8N+6))=e_{8N-9}(k-3,n-(16N+3))+a_{8N-11}(k-3,n-(16N+3)),
\end{equation*}
and
\begin{equation*}
e_{8N-2}(k-1,n-(8N+5))=e_{8N-9}(k-3,n-(16N+3))+a_{8N-11}(k-3,n-(16N+3)).
\end{equation*}
Therefore
\begin{equation*}
e_{8N-2}(k-1,n-(8N+5))=e_{8N-3}(k-2,n-(8N+6)).
\end{equation*}
So by summing~\eqref{pif1} and~\eqref{pif2} and replacing in~\eqref{pif3}, we get $$a_{8N+6}(k,n)=a_{8N+4}(k,n)+a_{8N}(k-1,n-(8N+5))+a_{8N-4}(k-2,n-(8N+6)),$$
which gives in terms of generating functions
\begin{equation*}
G_{8N+6}(t,q)=G_{8N+4}(t,q)+tq^{8N+5} G_{8N}(t,q)+t^{2}q^{8N+6}G_{8N-4}(t,q).
\end{equation*}
By~\eqref{eq34}, the results from the last two subsections and the induction hypothesis,
$$G_{8N+6}(t,q)=\left(1+tq\right)G_{8N+3}(tq^2,q).$$
This concludes the proof of Theorem~\ref{main}.

\subsection{Final argument}

By Theorem~\ref{main}, we have for all $N \in \N^*,$ $$G_{2N}(t,q)=(1+tq)G_{2N-3}(tq^2,q).$$

So, if we let $N \rightarrow \infty$, we obtain:

\begin{equation}
\label{iteration}
G_{\infty}(t,q)=\left(1+tq\right)G_{\infty}(tq^2,q).
\end{equation}

Iteration of \eqref{iteration} shows that:
\begin{equation*}
G_{\infty}(t,q)= \prod_{k=0}^{\infty} \left(1+ tq^{2k+1}\right).
\end{equation*}

This completes the proof of Theorem~\ref{refinement}.

\section{Conclusion}
We have proved combinatorially and refined Theorem~\ref{siladic}. It would be interesting to see if other partition identities arising from the theory of vertex operators or Lie algebras can be proved using similar methods. Papers by Siladi\'c~\cite{Siladic}, Primc~\cite{Primc} and Meurman-Primc~\cite{Meurman} contain examples of such identities.

Furthermore in~\cite{Alladi}, Alladi, Andrews and Gordon give a bijective proof and a refinement of Capparelli's conjecture, which also comes from the study of Lie algebras. One might investigate if a bijective proof would be possible for Theorem~\ref{siladic} too.

Finally, it would be a question of interest to determine if the variable $k$ of our refinement can also be interpreted algebraically.

\section*{Acknowledgements}
The author would like to thank Jeremy Lovejoy and Fr\'ed\'eric Jouhet for introducing her to this subject and sharing their ideas with her, and Jeremy Lovejoy for carefully reading the preliminary versions of this paper and giving her helpful suggestions to improve it.


\end{document}